\documentclass{amsart}
\usepackage{amsfonts}
\usepackage{mathrsfs}
\usepackage{amsmath}

\newtheorem{theorem}{Theorem}[section]
\newtheorem{lemma}[theorem]{Lemma}

\theoremstyle{definition}
\newtheorem{definition}[theorem]{Definition}

\newtheorem{prop}[theorem]{Proposition}
\newtheorem{cor}[theorem]{Corollary}

\theoremstyle{remark}
\newtheorem{remark}[theorem]{Remark}

\numberwithin{equation}{section}

\begin{document}

\title{Simplified and Equivalent Characterizations of Banach Limit Functional and Strong Almost Convergence}

\author{Chao You}
\address{Department of Mathematics $\&$ The Academy of Fundamental and Interdisciplinary Science\\Harbin Institute of Technology\\
Harbin 150001, Heilongjiang, China}
\email{hityou1982@gmail.com}



\dedicatory{This paper is dedicated to Prof. Lixin Xuan for his
great encouragement.}


\begin{abstract}
In this paper, we give simplified and equivalent characterizations
of Banach limit functional, which is the minimum requirement to
characterize strong almost convergence. With this machinery, we show
that Hajdukovi\'{c}'s\cite{Hajdukovic} quasi-almost convergence is
equivalent to strong almost convergence.
\end{abstract}

\maketitle

\section{Introduction}
Let $l^{\infty}$ be the Banach space of bounded sequences of real
numbers $x:=\{x(n)\}_{n=1}^{\infty}$ with supremum norm
$\|x\|_{\infty}:=\sup_n|x(n)|$. As an application of Hahn-Banach
theorem, a \emph{Banach limit} $L$ is a bounded linear functional on
$l^{\infty}$, which satisfies the following properties:

(i) If $x=\{x(n)\}_{n=1}^{\infty}\in l^{\infty}$ and $x(n)\geq 0$,
then $L(x)\geq 0$;

(ii) If $x=\{x(n)\}_{n=1}^{\infty}\in l^{\infty}$ and
$Tx=\{x(2),x(3),\ldots\}$, where $T$ is the \emph{left-shift
operator}, then $L(x)=L(Tx)$;

(iii) $\|L\|=1$;

(iv) If $x=\{x(n)\}_{n=1}^{\infty}\in c$, where $c$ is the Banach
subspace of $l^{\infty}$ consisting of convergent sequences, then
$L(x)=\lim_{n\rightarrow\infty}x(n)$.

Since the Hahn-Banach norm-preserving extension is not unique, there
must be many Banach limits in the dual space of $l^{\infty}$, and
usually different Banach limits have different values at the same
element in $l^{\infty}$. However, there indeed exist sequences whose
values of all Banach limits are the same. Condition (iv) is a
trivial example. Besides that, there also exist nonconvergent
sequences satisfying this property, for such examples please see
\cite{Feng} and \cite{You}. In \cite{Lorentz}, G. G. Lorentz called
a sequence $x=\{x(n)\}_{n=1}^{\infty}$ \emph{almost convergent}, if
all Banach limits of $x$, $L(x)$, are the same, and this unique
Banach limit is called \emph{F-limit} of $x$. In his paper, Lorentz
proved the following criterion for almost convergent sequences:

\begin{theorem}
A sequence $x=\{x(n)\}_{n=1}^{\infty}\in l^{\infty}$ is almost
convergent with F-limit $L(x)$ if and only if
$$
\lim_{n\rightarrow\infty}\frac{1}{n}\sum_{t=i}^{i+n-1}x(t)=L(x)
$$
uniformly in $i$.
\end{theorem}

Similar to this theorem, recently D. Hajdukovi\'{c}\cite{Hajdukovic}
and S. Shaw et al\cite{Shaw} generalized the concept of almost
convergence to bounded sequences in normed vector space and bounded
continuous vector-valued functions, respectively.

Suppose $(V,\|\cdot\|_V)$ is a complex normed vector space. Let
$l^{\infty}(V)$ be the normed vector space of bounded $V$-valued
sequences $x:=\{x_n\}^{\infty}_{n=1}$ with supremum norm
$\|x\|_{\infty}:=\sup_n\|x_n\|_V$. In particular, $c(V)$ is the
subspace of $l^{\infty}(V)$, which consists of convergent $V$-valued
sequences. For any $v\in V$, let $\widetilde{v}:=\{v,v,\ldots\}$
denote the sequence with constant entry $v$, clearly
$\widetilde{v}\in c(V)$.

\begin{definition}[\cite{Hajdukovic}]
Suppose $x=\{x_n\}^{\infty}_{n=1}\in l^{\infty}(V)$ and $v\in V$.
$\{x_n\}_{n=1}^{\infty}$ is called \emph{almost convergent} to $v$
if
$$\lim_{n\rightarrow\infty}\frac{1}{n}\sum_{i=0}^{n-1}x_{i+j}=v$$
uniformly in $j$.
\end{definition}

Let $C_b([0,\infty),V)$ be the normed vector space of bounded
$V$-valued continuous functions $f$ with supremum norm
$\|f\|:=\sup_{t\in[0,\infty)}\|f(t)\|_V$.

\begin{definition}[\cite{Shaw}]
Suppose $f\in C_b([0,\infty),V)$ and $v\in V$. $f(t)$ is called
\emph{almost convergent} to $v$ when $t\rightarrow\infty$ if
$$\lim_{t\rightarrow\infty}\frac{1}{t}\int_a^{a+t}f(s)ds=v$$
uniformly in $a$.
\end{definition}

In\cite{Hajdukovic}, Hajdukovi\'{c} also gave the concept of
\emph{quasi-almost convergence} in terms of some kind of linear
functionals, which are similar to Banach limit in the real sequence
case. First, Hajdukovi\'{c} defined a semi-norm $q$ on
$l^{\infty}(V)$ as following:

For $x=\{x_n\}^{\infty}_{n=1}\in l^{\infty}(V)$,
\begin{equation}
q(x)=\varlimsup_{n\rightarrow\infty}\left(\sup_j\frac{1}{n}\left\|\sum_{i=0}^{n-1}x_{i+jn}\right\|_V\right).
\end{equation}

And then, he showed that there exists the family $\Pi$ of nontrivial
linear functionals $L$ defined on $l^{\infty}(V)$ such that for all
$x=\{x_n\}^{\infty}_{n=1}\in l^{\infty}(V)$, the following
assertions are valid:

(i) $L(Tx)=L(x)$;

(ii) $|L(x)|\leq q(x)$;

(iii) $L(x-\widetilde{v})=0$ if and only if $q(x-\widetilde{v})=0$.

\begin{definition}
A sequence $x=\{x_n\}^{\infty}_{n=1}\in l^{\infty}(V)$ is called
\emph{quasi-almost convergent} to $v\in V$ if $\forall L\in \Pi$,
$L(x-\widetilde{v})=0$.
\end{definition}

Similar to the definition of almost convergence, Hajdukovi\'{c} gave
the following equivalent characterization of quasi-almost
convergence:

\begin{theorem}
Suppose $x=\{x_n\}^{\infty}_{n=1}\in l^{\infty}(V)$ and $v\in V$.
$\{x_n\}_{n=1}^{\infty}$ is quasi-almost convergent to $v$ if and
only if
$$\lim_{n\rightarrow\infty}\frac{1}{n}\sum_{i=0}^{n-1}x_{i+jn}=v$$
uniformly in $j$.
\end{theorem}

From this theorem, it seems that quasi-almost convergence is weaker
than almost convergence. However, in this paper, we will show that
actually they are equivalent! In section \ref{Banach limit
functional}, we will define the concept of \emph{Banach limit
functional}, which is a generalization of Banach limit in bounded
real sequence case but much simpler, even than Hajdukovi\'{c}'s
linear functionals $\Pi$. To show the existence and sufficiency of
Banach limit functionals, we provide a natural construction of
Banach limit functionals induced from $B_1(V^*)$. Then we will give
an equivalent characterization of Banach limit functional, which
shows that some items in traditional or Hajdukovi\'{c}'s definition
of Banach limit are equivalent or one could imply another, so it is
unnecessary to put them together in the definition.

In Section \ref{Strong almost convergence}, we define the concept of
\emph{strong almost convergence} in terms of Banach limit
functionals, and show that it is equivalent to almost convergence in
\cite{Hajdukovic}, then it is immediate that Hajdukovi\'{c}'s
quasi-almost convergence is equivalent to almost convergence too. We
also show that our almost convergence is stronger than that of J.
Kurtz's\cite{Kurtz}, so that's why we call it strong almost
convergence. Some basic properties of strong almost convergence are
also discussed. In particular, we show that though strong almost
convergence is weaker than norm convergence, corresponding
completenesses with respect to the two convergences are the same.

In the end, we point out that all definitions and results here could
be applied to bounded continuous functions exactly word by word from
summation to integration. Thus, to save space, we don't restate them
again.

\section{Banach Limit Functional of Bounded Sequences in Normed Vector
Space}\label{Banach limit functional}

\begin{definition}\label{Banach limit}
A bounded linear functional $L$ on $l^{\infty}(V)$ is called a
\emph{Banach limit functional} if it satisfies the following two
conditions:

(i) $\|L\|\leq 1$;

(ii) $\forall x=\{x_n\}^{\infty}_{n=1}\in l^{\infty}(V)$ and
$Tx=\{x_2,x_3,\ldots\}$, then $L(Tx)=L(x)$.
\end{definition}

To see the existence and sufficiency of Banach limit functionals,
let us begin with the following lemma, which is similar to that in
Sucheston's paper\cite{Sucheston}.

\begin{lemma}\label{q is well-defined}
$\forall x=\{x_n\}_{n=1}^{\infty}\in l^{\infty}(V)$,
$$\lim_{n\rightarrow\infty}\left(\sup_j\frac{1}{n}\left\|\sum_{i=0}^{n-1}x_{i+j}\right\|_V\right)$$
exists.
\end{lemma}

\begin{proof}
Set
$$c_n=\sup_j\frac{1}{n}\left\|\sum_{i=0}^{n-1}x_{i+j}\right\|_V.$$ We need to show that
$\lim_{n\rightarrow\infty}c_n$ exists. For each $m$, $n$, one has
\begin{align*}
\sup_j\left\|\sum_{i=0}^{m+n-1}x_{i+j}\right\|_V\leq&\sup_j\left(\left\|\sum_{i=0}^{m-1}x_{i+j}\right\|_V
+\left\|\sum_{i=m}^{m+n-1}x_{i+j}\right\|_V\right)\\
\leq&\sup_j\left\|\sum_{i=0}^{m-1}x_{i+j}\right\|_V+\sup_j\left\|\sum_{i=m}^{m+n-1}x_{i+j}\right\|_V\\
\leq&\sup_j\left\|\sum_{i=0}^{m-1}x_{i+j}\right\|_V+\sup_j\left\|\sum_{i=0}^{n-1}x_{i+j}\right\|_V,
\end{align*}
i.e., $(m+n)c_{m+n}\leq mc_m+nc_n$. Thus $$(r+km)c_{r+km}\leq
rc_r+kmc_{km}\leq rc_r+kmc_m.$$ Dividing by $r+km$ and letting
$k\rightarrow\infty$ with $r$, $m$ fixed, we obtain
$$\limsup_{k\rightarrow\infty}c_{r+km}\leq c_m.$$
Since this holds for $r=0,1,\ldots,m-1$,
$\limsup_{n\rightarrow\infty}c_n\leq c_m$ for each $m$, and hence
$\limsup_{n\rightarrow\infty}c_n\leq\liminf_{m\rightarrow\infty}c_m$,
which implies that $\lim_{n\rightarrow\infty}c_n$ exists.
\end{proof}

\begin{definition}
For any $x\in l^{\infty}(V)$, define
\begin{equation}
p(x)=\lim_{n\rightarrow\infty}\left(\sup_j\frac{1}{n}\left\|\sum_{i=0}^{n-1}x_{i+j}\right\|_V\right).
\end{equation}
\end{definition}

From Lemma \ref{q is well-defined}, it is easy to see that $p$ is a
well-defined seminorm on $l^{\infty}(V)$.

\begin{lemma}\label{xiangdeng}
If $x=\{x_n\}_{n=1}^{\infty}\in l^{\infty}(V)$ such that
$$\lim_{n\rightarrow\infty}\frac{1}{n}\left\|\sum_{i=0}^{n-1}x_{i+j}\right\|_V=m$$
uniformly in $j$, then
$$\lim_{n\rightarrow\infty}\left(\sup_j\frac{1}{n}\left\|\sum_{i=0}^{n-1}x_{i+j}\right\|_V\right)
=\lim_{n\rightarrow\infty}\frac{1}{n}\left\|\sum_{i=0}^{n-1}x_{i+j}\right\|_V=m.$$
\end{lemma}

\begin{proof}
$\forall \varepsilon>0$, since
$$\lim_{n\rightarrow\infty}\frac{1}{n}\left\|\sum_{i=0}^{n-1}x_{i+j}\right\|_V=m$$
uniformly in $j$, there exists $N\in \mathbb{N}$ such that for any
$j\in \mathbb{N}$ if $n>N$, then
$$\left|\frac{1}{n}\left\|\sum_{i=0}^{n-1}x_{i+j}\right\|_V-m\right|<\varepsilon,$$
i.e.,
$$m-\varepsilon<\frac{1}{n}\left\|\sum_{i=0}^{n-1}x_{i+j}\right\|_V<m+\varepsilon.$$
Hence
$$m-\varepsilon<\sup_j\frac{1}{n}\left\|\sum_{i=0}^{n-1}x_{i+j}\right\|_V\leq
m+\varepsilon.$$ Since $\varepsilon$ is arbitrary, it follows that
$$\lim_{n\rightarrow\infty}\left(\sup_j\frac{1}{n}\left\|\sum_{i=0}^{n-1}x_{i+j}\right\|_V\right)=m.$$
\end{proof}

\begin{lemma}\label{zero}
If $x=\{x_n\}_{n=1}^{\infty}\in c(V)$ such that
$\lim_{n\rightarrow\infty}x_n=0$, then
$$\lim_{n\rightarrow\infty}\frac{1}{n}\left\|\sum_{i=0}^{n-1}x_{i+j}\right\|_V=0$$
uniformly in $j$.
\end{lemma}

\begin{proof}
$\forall \varepsilon>0$, since $\lim_{n\rightarrow\infty}x_n=0$,
there exists $N_1\in \mathbb{N}$ such that $\|x_n\|_V<\varepsilon/2$
if $n>N_1$. Choose $N_2$ such that
$(\|x_1\|_V+\|x_2\|_V+\cdots+\|x_{N_1}\|_V)/N_2<\varepsilon/2$. Let
$N=\max\{N_1,N_2\}$. Let $n>N$, for any $j\in\mathbb{N}$, if
$j>N_1$, then
$$\frac{1}{n}\left\|\sum_{i=0}^{n-1}x_{i+j}\right\|_V\leq\frac{\sum_{i=0}^{n-1}\|x_{i+j}\|_V}{n}
<\frac{n\varepsilon/2}{n}=\varepsilon/2;$$ if $j\leq N_1$,
$$
\frac{1}{n}\left\|\sum_{i=0}^{n-1}x_{i+j}\right\|_V\leq\frac{\sum_{i=0}^{N_1-j}\|x_{i+j}\|_V+\sum_{i=N_1-j+1}^{n-1}\|x_{i+j}\|_V}{n}
<\varepsilon/2+\varepsilon/2=\varepsilon.
$$
Hence
$$\lim_{n\rightarrow\infty}\frac{1}{n}\left\|\sum_{i=0}^{n-1}x_{i+j}\right\|_V=0$$
uniformly in $j$.
\end{proof}

\begin{cor}\label{jixian}
If $x=\{x_n\}_{n=1}^{\infty}\in c(V)$ with
$\lim_{n\rightarrow\infty}x_n=v\in V$, then
$$\lim_{n\rightarrow\infty}\frac{1}{n}\left\|\sum_{i=0}^{n-1}x_{i+j}\right\|_V=\|v\|_V$$
uniformly in $j$.
\end{cor}

\begin{proof}
Since $\lim_{n\rightarrow\infty}x_n=v$, i.e.,
$\lim_{n\rightarrow\infty}(x_n-v)=0$, it follows from Lemma
\ref{zero} that
$$\lim_{n\rightarrow\infty}\frac{1}{n}\left\|\sum_{i=0}^{n-1}x_{i+j}-nv\right\|_V=0$$
uniformly in $j$. Since
$$
\left|\frac{1}{n}\left\|\sum_{i=0}^{n-1}x_{i+j}\right\|_V-\|v\|_V\right|=
\frac{1}{n}\left|\left\|\sum_{i=0}^{n-1}x_{i+j}\right\|_V-\|nv\|_V\right|\leq
\frac{1}{n}\left\|\sum_{i=0}^{n-1}x_{i+j}-nv\right\|_V,
$$
it follows that
$$\lim_{n\rightarrow\infty}\frac{1}{n}\left\|\sum_{i=0}^{n-1}x_{i+j}\right\|_V=\|v\|_V$$
uniformly in $j$.
\end{proof}

\begin{definition}
Suppose that $f\in V^*$ and $\|f\|\leq1$, define the \emph{induced
bounded linear functional $L_f$} on $c(V)$ as following: for any
$x=\{x_n\}\in c(V)$ with $\lim_{n\rightarrow\infty}x_n=v\in V$,
$L_f(x)=f(v)$.
\end{definition}

\begin{prop}\label{norm control2}
For any $x=\{x_n\}_{n=1}^{\infty}\in c(V)$ with
$\lim_{n\rightarrow\infty}x_n=v\in V$, $|L_f(x)|\leq p(x)$.
\end{prop}

\begin{proof}
From Lemma \ref{xiangdeng} and Corollary \ref{jixian}, we have
$$
|L_f(x)|=|f(v)|\leq\|v\|_V=\lim_{n\rightarrow\infty}\frac{1}{n}\left\|\sum_{i=0}^{n-1}x_{i+j}\right\|_V
=\lim_{n\rightarrow\infty}\left(\sup_j\frac{1}{n}\left\|\sum_{i=0}^{n-1}x_{i+j}\right\|_V\right)=p(x).
$$
\end{proof}

\begin{cor}
$\|L_f\|\leq1$.
\end{cor}

\begin{proof}
$\forall x=\{x_n\}_{n=1}^{\infty}\in c(V)$, from Proposition
\ref{norm control2}, $|L_f(x)|\leq p(x)\leq \|x\|_{\infty}$. So
$\|L_f\|\leq1$.
\end{proof}

From Hahn-Banach Theorem, we know that there must exist a
norm-preserving extension $\overline{L}_f$ of $L_f$ on whole
$l^{\infty}(V)$ such that
\begin{equation}\label{norm control}
|\overline{L}_f(x)|\leq p(x),
\end{equation}
$\forall x=\{x_n\}_{n=1}^{\infty}\in l^{\infty}(V)$. Now we will
show that such $\overline{L}_f$ is an example of Banach limit
functional as defined in Definition \ref{Banach limit}.

\begin{theorem}\label{shift invariant}
If $L\in l^{\infty}(V)^*$ and $x=\{x_n\}_{n=1}^{\infty}\in
l^{\infty}(V)$ such that $|L(x)|\leq p(x)$, then $L(Tx)=L(x)$.
\end{theorem}

\begin{proof}
Define sequence $y:=\{y_n\}_{n=1}^{\infty}$ as $y_n:=x_{n+1}-x_n$,
i.e., $y=Tx-x$. Since $x$ is bounded, $y$ is also bounded, i.e.,
$y\in l^{\infty}(V)$. Then we have
\begin{align*}
p(y)=&\lim_{n\rightarrow\infty}\left(\sup_j\frac{1}{n}\left\|\sum_{i=0}^{n-1}(x_{i+j+1}-x_{i+j})\right\|_V\right)\\
=&\lim_{n\rightarrow\infty}\left(\sup_j\frac{1}{n}\left\|x_{n+j}-x_j)\right\|_V\right)\\
\leq&\lim_{n\rightarrow\infty}\frac{2\|x\|_{\infty}}{n}=0.
\end{align*}
Since $ |L(y)|\leq p(y)=0,$ i.e., $L(y)=0$, we have
$$L(y)=L(Tx-x)=L(Tx)-L(x)=0,$$ i.e., $L(Tx)=L(x)$.
\end{proof}

So far, we have shown that $\overline{L}_f$ is indeed a Banach limit
functional. Since that $f$ is an arbitrary choice from $B_1(V^*)$
and Hahn-Banach norm-preserving extension is not unique, we can see
that $l^{\infty}(V)$ has sufficiently many Banach limit functionals.
Let us denote all the Banach limit functionals of $l^{\infty}(V)$ by
$\mathfrak{L}(V)$.

\begin{remark}
Our definition of Banach limit functional here has greatly improved
and simplified corresponding definition in D. Hajdukovi\'{c}'s paper
\cite{Hajdukovic}. First of all, you will find that we don't confine
$V$ to be only real normed vector space. Actually, since there is no
longer positive element in normed vector space, we don't need real
scalars. And we also improve the definition of $p(x)$ from
$\varlimsup_{n\rightarrow\infty}\left(\sup_j\frac{1}{n}\left\|\sum_{i=0}^{n-1}x_{i+j}\right\|_V\right)
$ to
$\lim_{n\rightarrow\infty}\left(\sup_j\frac{1}{n}\left\|\sum_{i=0}^{n-1}x_{i+j}\right\|_V\right)$,
which is more accurate. Moreover, due to the following theorem, we
will show that suppose $L\in l^{\infty}(V)^*$ and $\|L\|\leq 1$, for
any $x=\{x_n\}_{n=1}^{\infty}\in l^{\infty}(V)$,
$L(Tx)=L(x)\Longleftrightarrow |L(x)|\leq p(x)$. Hence we exclude
the condition $|L(x)|\leq p(x)$ from Definition \ref{Banach limit}.
\end{remark}

\begin{theorem}\label{equiv}
Suppose $L\in l^{\infty}(V)^*$, the following two statements are
equivalent:

(i)L is a Banach limit functional;

(ii)$|L(x)|\leq p(x)$, $\forall x=\{x_n\}_{n=1}^{\infty}\in
l^{\infty}(V)$.
\end{theorem}

\begin{proof}
(ii)$\Longrightarrow$(i) is exactly Theorem \ref{shift invariant}.

For (i)$\Longrightarrow$(ii), let $c_n=\{c_{n,j}\}_{j=1}^{\infty}\in
l^{\infty}(V)$, where $c_{n,j}=\frac{1}{n}\sum_{i=0}^{n-1}x_{i+j}$,
i.e., $c_n=\frac{1}{n}\sum_{i=0}^{n-1}T^ix$. Then for any Banach
limit functional $L$, we have
$$\sup_j\frac{1}{n}\left\|\sum_{i=0}^{n-1}x_{i+j}\right\|_V=\|c_n\|_{\infty}\geq|L(c_n)|
=|L(\frac{1}{n}\sum_{i=0}^{n-1}T^ix)|=\frac{1}{n}|\sum_{i=0}^{n-1}L(T^ix)|=|L(x)|.$$
Hence $|L(x)|\leq
\lim_{n\rightarrow\infty}\left(\sup_j\frac{1}{n}\left\|\sum_{i=0}^{n-1}x_{i+j}\right\|_V\right)=p(x)$.
\end{proof}

Classical Banach limit on bounded real sequences is a generalization
of ordinary convergence, so item (iv) is always included in the
definition of Banach limit. From our view point of Banach limit
functional here, Banach limit actually is the Banach limit
functional induced from linear functional $f(x)=x$, $\forall x\in
\mathbb{R}$. Moreover, the following proposition shows that in some
sense item (iv) can be implied from item (ii) and (iii). Hence, our
definition of Banach limit functional is essentially a
simplification of Banach limit.

\begin{prop}\label{generalize}
If $L\in \mathfrak{L}(V)$ and $x=\{x_n\}^{\infty}_{n=1}\in c(V)$
with $\lim_{n\rightarrow\infty}x_n=v\in V$, then
$L(x)=L(\widetilde{v})$.
\end{prop}

\begin{proof}
Since $\lim_{n\rightarrow\infty}(x_n-v)=0$, it follows from Lemma
\ref{xiangdeng} and Lemma \ref{zero} that $p(x-\widetilde{v})=0$.
From Theorem \ref{equiv}, $|L(x-\widetilde{v})|\leq
p(x-\widetilde{v})=0$, i.e., $L(x)=L(\widetilde{v})$.
\end{proof}

\begin{remark}
Before finishing this section, we remark that in the definition of
classical Banach limit of bounded real sequences, item
(i)(positivity) could be implied from item (iii) and (iv), so this
item could be excluded. Moreover, due to Proposition
\ref{generalize}, item (iv) could be replaced by (iv')
$L(\widetilde{1})=1$. We leave the proofs as easy exercises to
interested readers.
\end{remark}

\section{Strong Almost Convergence of Bounded Sequences in Normed Vector
Space}\label{Strong almost convergence}

\begin{definition}
A sequence $x=\{x_n\}_{n=1}^{\infty}\in l^{\infty}(V)$ is called
\emph{strongly almost convergent} to $v\in V$ if for any Banach
limit functional $L\in \mathfrak{L}(V)$, it holds that
$L(x)=L(\widetilde{v})$. Let us denote it by
$x_n\overset{s.a.}{\longrightarrow} v$, and $v$ is called
\emph{strong almost limit} of $x$.
\end{definition}

Next we will give an equivalent characterization of strong almost
convergence, and show that our strong almost convergence is
equivalent to almost convergence given by
Hajdukovi\'{c}\cite{Hajdukovic}. Moreover, as an immediate
corollary, his quasi-almost convergence is equivalent too.

\begin{lemma}\label{vanish}
Suppose $x=\{x_n\}_{n=1}^{\infty}\in l^{\infty}(V)$. $p(x)=0$ if and
only if $L(x)=0$, $\forall L\in \mathfrak{L}(V)$.
\end{lemma}

\begin{proof}
If $p(x)=0$, then $\forall L\in \mathfrak{L}(V)$, it follows from
Theorem \ref{equiv} that $|L(x)|\leq p(x)=0$. Hence $L(x)=0$.

Conversely. Since $\forall L\in \mathfrak{L}(V)$ $L(x)=0$, it
suffices to find a particular Banach limit functional $L_0$ such
that $L_0(x)=p(x)$. The following is the construction of such $L_0$.
Let $M=\{\lambda x: \lambda\in \mathbb{C}\}$ be a subspace of
$l^{\infty}(V)$. On $M$ define $f_0(\lambda x)=\lambda p(x)$, then
$|f_0(y)|=p(y)$ $\forall y\in M$. From Hahn-Banach Theorem, we can
get an extension $L_0$ of $f_0$ on whole $l^{\infty}(V)$ such that
$|L_0(y)|\leq p(y)$ $\forall y\in l^{\infty}(V)$. From Theorem
\ref{equiv}, we know that $L_0$ is a Banach limit functional. So we
are done.
\end{proof}

\begin{remark}
In fact, so far all statements concerning $p$ in Section \ref{Banach
limit functional} and Section \ref{Strong almost convergence} still
hold for $q$, so we can see that quasi-almost convergence given by
Hajdukovi\'{c} is actually equivalent to strong almost convergence.
\end{remark}

An immediate corollary of Lemma \ref{vanish} is the following
important theorem:
\begin{theorem}\label{s.a.c.}
A sequence $x=\{x_n\}_{n=1}^{\infty}\in l^{\infty}(V)$ is strongly
almost convergent to $v\in V$ if and only if $p(x-\widetilde{v})=0$.
\end{theorem}

\begin{prop}\label{basic properties of s.a.}
(i) If $x=\{x_n\}_{n=1}^{\infty}\in l^{\infty}(V)$ is strongly
almost convergent in $V$, then its strong almost limit is unique.

(ii) Suppose $x=\{x_n\}_{n=1}^{\infty}, y=\{y_n\}_{n=1}^{\infty}\in
l^{\infty}(V)$. If $x_n\overset{s.a.}{\longrightarrow} u$ and
$y_n\overset{s.a.}{\longrightarrow}v$, then for any
$\lambda,\mu\in\mathbb{C}$, $\lambda x_n+\mu
y_n\overset{s.a.}{\longrightarrow}\lambda u+\mu v$.

(iii) If $\{x_n\}_{n=1}^{\infty}$ is a sequence from $V$ such that
$\lim_{n\rightarrow\infty}x_n=v\in V$, then
$x_n\overset{s.a.}{\longrightarrow} v$.
\end{prop}

\begin{proof}
(i) If $x_n\overset{s.a.}{\longrightarrow} v_1$ and
$x_n\overset{s.a.}{\longrightarrow} v_2$ simultaneously, then it
follows from Theorem \ref{s.a.c.} that $\|v_1-v_2\|_V=p(v_1-v_2)\leq
p(\widetilde{v_1}-x)+p(x-\widetilde{v_2})=0+0=0$. Hence $v_1=v_2$.

(ii) $p(\lambda x+\mu y-\lambda \widetilde{u}-\mu \widetilde{v})\leq
|\lambda| p(x-\widetilde{u})+|\mu|p(y-\widetilde{v})$.

(iii) From Theorem \ref{generalize}.
\end{proof}

\begin{remark}
Please notice that if $x_n\overset{s.a.}{\longrightarrow}v\in V$, it
doesn't mean that each subsequence of $x=\{x_n\}_{n=1}^{\infty}$ is
also strongly almost convergent, let alone strongly almost
convergent to the same vector. For example, consider bounded real
sequence $x=\{1,0,1,0,\ldots\}$. Then
$x_n\overset{s.a.}{\longrightarrow}1/2$. However,
$\lim_{k\rightarrow\infty}x_{2k-1}=1$, while
$\lim_{k\rightarrow\infty}x_{2k}=0$.
\end{remark}

\begin{lemma}\label{zero equaility}
Suppose $x=\{x_n\}_{n=1}^{\infty}\in l^{\infty}(V)$ and $p(x)=0$,
then
$$\lim_{n\rightarrow\infty}\frac{1}{n}\left\|\sum_{i=0}^{n-1}x_{i+j}\right\|_V=0$$
uniformly in $j$.
\end{lemma}

\begin{proof}
Since $p(x)=0$, for any $\varepsilon>0$, there exists $N\in
\mathbb{N}$ such that
$$\sup_j\frac{1}{n}\left\|\sum_{i=0}^{n-1}x_{i+j}\right\|_V<\varepsilon$$
when $n>N$. In other words, for any $j\in \mathbb{N}$, when $n>N$
$$\frac{1}{n}\left\|\sum_{i=0}^{n-1}x_{i+j}\right\|_V<\varepsilon.$$
So
$$\lim_{n\rightarrow\infty}\frac{1}{n}\left\|\sum_{i=0}^{n-1}x_{i+j}\right\|_V=0$$
uniformly in $j$.
\end{proof}

\begin{theorem}\label{uniformly averagely convergent}
Suppose $x=\{x_n\}_{n=1}^{\infty}\in l^{\infty}(V)$.
$x_n\overset{s.a.}{\longrightarrow}v\in V$ if and only if
$$\lim_{n\rightarrow\infty}\frac{1}{n}\sum_{i=0}^{n-1}x_{i+j}=v$$
uniformly in $j$.
\end{theorem}

\begin{proof}
$x_n\overset{s.a.}{\longrightarrow}v\Longleftrightarrow
p(x-\widetilde{v})=0$. From Lemma \ref{zero equaility},
$$\lim_{n\rightarrow\infty}\frac{1}{n}\left\|\sum_{i=0}^{n-1}(x_{i+j}-v)\right\|_V
=\lim_{n\rightarrow\infty}\left\|\frac{1}{n}\sum_{i=0}^{n-1}x_{i+j}-v\right\|_V=0$$
uniformly in $j$, i.e.,
$$\lim_{n\rightarrow\infty}\frac{1}{n}\sum_{i=0}^{n-1}x_{i+j}=v$$
uniformly in $j$.
\end{proof}

This theorem shows that strong almost convergence is equivalent to
almost convergence in \cite{Hajdukovic}, and so is quasi-almost
convergence.

\begin{remark}
In the definition of strong almost convergence, we require
$x=\{x_n\}_{n=1}^{\infty}$ to be bounded. Actually this is not
constrained, because from
$$\lim_{n\rightarrow\infty}\frac{1}{n}\sum_{i=0}^{n-1}x_{i+j}=v$$
uniformly in $j$, we can easily imply that $\{x_n\}_{n=1}^{\infty}$
is bounded.
\end{remark}

\begin{cor}\label{convex}
Suppose $x=\{x_n\}_{n=1}^{\infty}\in l^{\infty}(V)$. If
$x_n\overset{s.a.}{\longrightarrow}v\in V$, then $v\in
\overline{co}\{x_n:n\in\mathbb{N}\}$.
\end{cor}

\begin{definition}
$V$ is a normed vector space and $A\subseteq V$. $A$ is called
\emph{s.a.-sequentially closed} if $\forall \{x_n\}_{n=1}^{\infty}$
from $A$ such that $x_n\overset{s.a.}{\longrightarrow}v\in V$, then
$v\in A$.
\end{definition}

\begin{theorem}
Suppose $V$ is a normed vector space and $A\subseteq V$ is convex.
$A$ is (norm) closed if and if $A$ is s.a.-sequentially closed. In
particular, a subspace of $V$ is (norm) closed if and if it is
s.a.-sequentially closed.
\end{theorem}

\begin{proof}
Suppose $A$ is s.a.-sequentially closed. If
$\{x_n\}_{n=1}^{\infty}\subseteq A$ and
$\lim_{n\rightarrow\infty}x_n=v\in V$. From Theorem \ref{basic
properties of s.a.} (iii), $x_n\overset{s.a.}{\longrightarrow}v\in
A$. Hence $A$ is (norm) closed.

Conversely, suppose $A$ is (norm) closed. If
$\{x_n\}_{n=1}^{\infty}\subseteq A$ and
$x_n\overset{s.a.}{\longrightarrow}v\in V$. From Corollary
\ref{convex}, $v\in \overline{co}\{x_n:n\in\mathbb{N}\}\subseteq
\overline{A}=A$. Hence $A$ is s.a.-sequentially closed.
\end{proof}

\begin{definition}
A bounded sequence $\{x_n\}_{n=1}^{\infty}$ of normed vector space
$V$ is called an \emph{s.a.-Cauchy sequence} if for any
$\varepsilon>0$ there exists $n\in \mathbb{N}$ such that for any
$j\in \mathbb{N}$ if $n,m>N$, then
$\left\|\frac{1}{n}\sum_{i=0}^{n-1}x_{i+j}-\frac{1}{m}\sum_{i=0}^{m-1}x_{i+j}\right\|_V<\varepsilon$.
$V$ is called \emph{s.a.-complete} if every s.a.-Cauchy sequence in
$V$ is strongly almost convergent to a vector in $V$.
\end{definition}

\begin{cor}
A normed vector space $V$ is (norm) complete if and only if it is
s.a.-complete.
\end{cor}

\begin{remark}
This shows that though strong almost convergence is weaker than
(norm) convergence, considering completion, it doesn't enlarge the
space further.
\end{remark}

In the end, we will explain why we use the terminology strong almost
convergence.
\begin{definition}[J. Kurtz\cite{Kurtz}]
Suppose $x=\{x_n\}_{n=1}^{\infty}\in l^{\infty}(V)$. We say that
$x=\{x_n\}_{n=1}^{\infty}$ is \emph{weakly almost convergent} to
$v\in V$ if for any $f\in V^*$,
$\widehat{f}(x):=\{f(x_n)\}_{n=1}^{\infty}\in
l^{\infty}(\mathbb{C})$ is almost convergent to $f(v)$. Let us
denote it by $x_n\overset{w.a.}{\longrightarrow}v$.
\end{definition}

\begin{remark}
From the definition, it is immediate that any weakly convergent
sequence is weakly almost convergent to its weak limit.
\end{remark}

\begin{theorem}
Suppose $x=\{x_n\}_{n=1}^{\infty}\in l^{\infty}(V)$ and $v\in V$. If
$x_n\overset{s.a.}{\longrightarrow}v$, then
$x_n\overset{w.a.}{\longrightarrow} v$.
\end{theorem}

\begin{proof}
From Theorem \ref{s.a.c.}, we just need to show that for any $f\in
V^*$, $p(\widehat{f}(x)-\widetilde{f(v)})=0$. Since
$p(x-\widetilde{v})=0$, we have
\begin{align*}
p(\widehat{f}(x)-\widetilde{f(v)})=&\lim_{n\rightarrow\infty}\left(\sup_j\frac{1}{n}\left|\sum_{i=1}^{n-2}(f(x_{i+j})-f(v))\right|\right)\\
=&\lim_{n\rightarrow\infty}\left(\sup_j\frac{1}{n}\left|f\left(\sum_{i=1}^{n-2}(x_{i+j}-v)\right)\right|\right)\\
\leq&\lim_{n\rightarrow\infty}\left(\sup_j\frac{1}{n}\|f\|\left\|\sum_{i=1}^{n-2}(x_{i+j}-v)\right\|_V\right)\\
=&\|f\|p(x-\widetilde{v})=0.
\end{align*}
\end{proof}

\begin{theorem}[J. Kurtz\cite{Kurtz}]
Suppose $x=\{x_n\}_{n=1}^{\infty}\in l^{\infty}(V)$ and $v\in V$. If
$\{x_n: n\in\mathbb{N}\}$ is precompact and
$x_n\overset{w.a.}{\longrightarrow} v$, then
$x_n\overset{s.a.}{\longrightarrow}v$.
\end{theorem}

\begin{remark}
When $V=\mathbb{C}$, strong almost convergence and weak almost
convergence coincide, since each bounded sequence in $\mathbb{C}$ is
precompact. So we just say almost convergence there.
\end{remark}

\bibliographystyle{amsplain}

\end{document}